\documentclass[11pt]{amsart}

\usepackage{amssymb,graphics,epsfig,overpic}
\usepackage[usenames,dvipsnames]{color}
\usepackage[colorlinks=true,linkcolor=blue,citecolor=BrickRed]{hyperref}
\usepackage{wrapfig,comment}

\newtheorem{thm}{Theorem}[section]

\newtheorem{lem}[thm]{Lemma}
\newtheorem{cor}[thm]{Corollary}
\newtheorem{prop}[thm]{Proposition}

\theoremstyle{definition}

\newtheorem{note}[thm]{Note}

\newcommand{\R}{\mathbf{R}}

\newcommand{\ol}{\overline}
\newcommand{\C}{\mathcal{C}}
\newcommand{\p}{\partial}

\newcommand{\lan}{\langle}
\newcommand{\ran}{\rangle}

\renewcommand{\S}{\mathbf{S}}
\renewcommand{\l}{\langle}
\renewcommand{\r}{\rangle}
\renewcommand{\(}{\left(}
\renewcommand{\)}{\right)}
\renewcommand{\tilde}{\widetilde}

\DeclareMathOperator{\dist}{dist}

\DeclareMathOperator{\conv}{conv}

\DeclareMathOperator{\Jac}{Jac}

\title[]{The length, width, and inradius of space curves}

\author{Mohammad Ghomi}
\address{School of Mathematics, Georgia Institute of Technology,
Atlanta, GA 30332}
\email{ghomi@math.gatech.edu}
\urladdr{www.math.gatech.edu/$\sim$ghomi}

\date{\today \,(Last Typeset)}
\subjclass[2000]{Primary: 53A04,  52A40; Secondary: 52A38, 52A15.}
\keywords{Width, inradius, convex hull, second hull, Crofton's formulas, sphere inspection, escape path, asteroid survey, optimal search pattern, Bellman's  lost in a forest problem.}
\thanks{Research of the  author was supported in part by NSF Grant DMS--1308777.}

\begin{document}

%\vspace*{-0.5in}

\begin{abstract} 
The width $w$ of a curve $\gamma$ in Euclidean space $\R^n$ is the infimum of the distances between all pairs of parallel hyperplanes which bound $\gamma$, while its inradius $r$ is the supremum of the radii of all spheres which are contained in the convex hull of $\gamma$ and are disjoint from $\gamma$. We use a mixture of topological and integral geometric techniques, including an application of Borsuk Ulam theorem due to Wienholtz and Crofton's formulas, to obtain lower bounds on the length of $\gamma$ subject to constraints on $r$ and $w$.    The special case of closed curves is also considered in each category. Our estimates confirm some conjectures of Zalgaller up to $99\%$ of their stated value, while we also disprove one of them.
\end{abstract}

\maketitle

%\vspace{-0.3in}

\tableofcontents

\section{Introduction}

What is the smallest length of wire which  can be bent into a shape that never falls through the gap behind a desk? 
What is the shortest orbit which allows a satellite to survey  a spherical  asteroid? These are  well-known open problems \cite{zalgaller:1996,orourkeMO,cfg:book,hiriart:2008} in classical geometry of space curves  $\gamma\colon[a,b]\to\R^3$, which are concerned with minimizing the length $L$ of $\gamma$ subject to constraints on its \emph{width} $w$ and \emph{inradius} $r$ respectively. Here  $w$  is the infimum of the distances between all pairs of parallel planes which bound $\gamma$, while $r$ is the supremum of the radii of all spheres which are contained in the convex hull of $\gamma$ and are disjoint from $\gamma$.   In 1994--1996 Zalgaller  \cite{zalgaller:1994, zalgaller:1996} conjectured four explicit solutions to these problems, including the cases where $\gamma$ is restricted to be closed, i.e., $\gamma(a)=\gamma(b)$. In this work we  confirm Zalgaller's conjectures between $83\%$ and $99\%$ of their stated value, while we also find a counterexample to one of them. Our estimates for the width problem are as follows:

\begin{thm}\label{thm:main1}
For any rectifiable curve $\gamma\colon[a,b]\to\R^3$,
\begin{equation}\label{eq:0}
\frac{L}{w}\geq3.7669.
\end{equation}
Furthermore if $\gamma$ is closed, 
\begin{equation}\label{eq:1}
\frac{L}{w}\geq \sqrt{\pi^2+16}> 5.0862.
\end{equation}
\end{thm}

  In \cite{zalgaller:1994} Zalgaller constructs a curve, ``$L_3$", with $L/w\leq 3.9215$. Thus \eqref{eq:0} is better than $96\%$ sharp (since $3.7669/3.9215\geq 0.9605$). Further,
in Section \ref{sec:example} we will construct a closed cylindrical curve (Figure \ref{fig:cylinderandsphere}(a)) with $L/w< 5.1151$, which shows  that \eqref{eq:1} is at least $99.43\%$ sharp. In particular, the length of the shortest closed curve of width $1$ is approximately $5.1$.
 It has been known since Barbier in 1860 \cite{barbier:1860}, and follows from the Cauchy-Crofton formula (Lemma \ref{lem:cc}), that for \emph{closed} planar curves $L/w\geq\pi$, where equality holds only for curves of constant width. The corresponding question for \emph{general} planar curves, however, was answered only in 1961 when  Zalgaller \cite{zalgaller:1961} produced a caliper shaped curve (Figure \ref{fig:calipers}(a)) with $L/w\approx 2.2782$, which has been subsequently rediscovered several times  \cite{adhikari&pitman,klotzler:1987}; see \cite{alexander:2009,finch&wetzel}. In 1994, Zalgaller \cite{zalgaller:1994} studied the width problem for curves in $\R^3$, and produced a closed  curve,  ``$L_5$", which he claimed to be minimal;  however,   our cylindrical example in Section \ref{sec:example} improves upon Zalgaller's curve.
 
  \begin{figure}[h]
   \centering
    \begin{overpic}[height=1.25in]{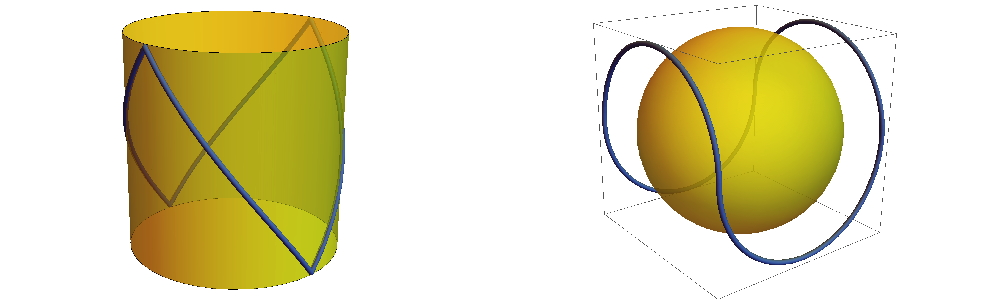} 
    \put(21.5,-3){\small(a)}
    \put(75,-2){\small(b)}
        \end{overpic}
   \caption{}
   \label{fig:cylinderandsphere}
\end{figure}

Next we describe our estimates for the inradius problem. Obviously $w\geq 2r$, and thus Theorem \ref{thm:main1} immediately yields   $L/r\geq 7.5338$ for general curves, and $L/r\geq 10.1724$ for closed curves. Using different techniques, however, we will improve these estimates as follows:

\begin{thm}\label{thm:main2}
For any rectifiable curve $\gamma\colon[a,b]\to\R^3$,
\begin{equation}\label{eq:2}
 \frac{L}{r}\geq \sqrt{(\pi+2)^2+36}> 7.9104.
\end{equation}
Furthermore if $\gamma$ is closed, 
\begin{equation}\label{eq:3}
\frac{L}{r}\geq 6\sqrt 3> 10.3923.
\end{equation}
\end{thm}

  In \cite[Sec. 2.12]{zalgaller:2003}, Zalgaller constructs a spiral curve with $L/r\leq 9.5767$, which shows that \eqref{eq:2} is better than $82.6\%$ optimal. Further, in \cite{zalgaller:1996}, he produces a curve composed of four semicircles (Figure \ref{fig:cylinderandsphere}(b)) with $L/r=4\pi$; see also \cite{orourkeMO} where this ``basebal stitches" is rediscovered  in 2011. Thus we may say that \eqref{eq:3} is better than $82.69\%$ optimal. For planar curves, the inradius  problem is nontrivial only for open arcs, and the answer, which is a horseshoe shaped curve (Figure \ref{fig:calipers}(b)) with $L/r=\pi+2$, was obtained  in 1980 by Joris \cite{joris},   see also  \cite[Sec. A30]{cfg:book}, \cite{faber-mycielski:1986,finch&wetzel,eggleston:1982}. Zalgaller studied the inradius problem for space arcs in 1994 \cite{zalgaller:1994} and for closed space curves in 1996 \cite{zalgaller:1996}. The latter problem also appears in Hiriart-Urruty \cite{hiriart:2008}.

 Both the width and inradius problems may be traced back to a 1956 question of Bellman  \cite{bellman:1956}: how long is the shortest escape path for a random point (lost hiker) inside an infinite parallel strip (forest) of known width? See  \cite{finch&wetzel} for more on these types of problems. Our width problem is the analogue of Bellman's question in $\R^3$.
   The inradius problem also has an intuitive reformulation known as the  ``sphere inspection" \cite{zalgaller:2003,orourkeMO} or the  ``asteroid surveying'' problem \cite{chan&golynski:2003}. 
   To describe this variation, let us say that a space curve $\gamma$ \emph{inspects} the sphere $\S^2$, or is an \emph{inspection curve}, provided that $\gamma$ lies outside $\S^2$ and for each point $x$ of $\S^2$ there exists a point $y$ of $\gamma$ such that the line segment $xy$ does not enter $\S^2$ (in other words, $x$ is ``visible" from $y$). It is easy to see that $\gamma$ inspects $\S^2$, after a translation, if and only if its inradius is $1$ \cite[p. 369]{zalgaller:1996}. Thus finding the shortest inspection curve  is equivalent to the inradius problem for $r=1$.   
   
   \begin{figure}[h]
   \centering
    \begin{overpic}[height=1in]{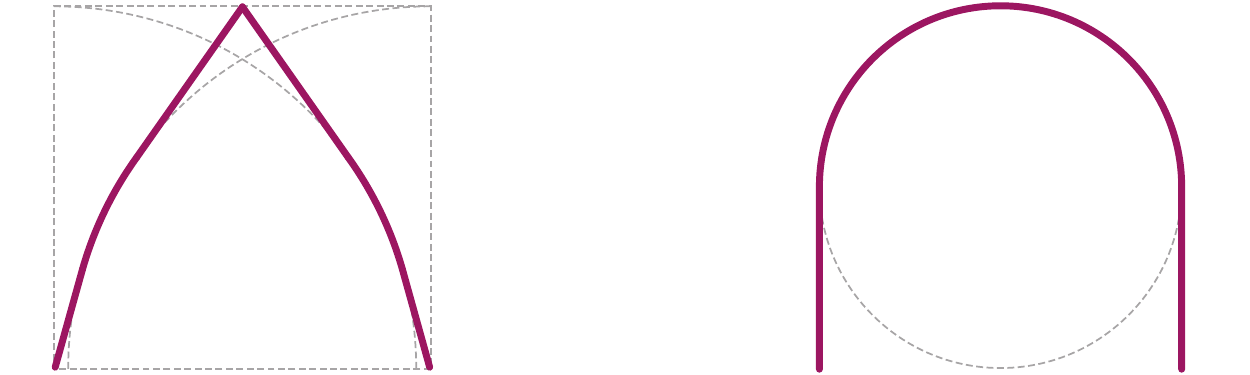} 
    \put(17.5,-4){\small(a)}
    \put(78,-4){\small(b)}
        \end{overpic}
   \caption{}
   \label{fig:calipers}
\end{figure}

 The proof of Theorem \ref{thm:main1} is based on an unpublished result of Daniel Wienholtz \cite{wiehnoltz:parallel}, which we include in Section \ref{sec:wienholtz}. This remarkable observation, which follows from Borsuk-Ulam theorem, states that any closed space curve has a pair of parallel support planes $H_0$, $H_1$ which intersect the curve at least twice each in alternating fashion 
 (Theorem \ref{thm:wienholtz}). Further, it is easy to see that the same result holds for general curves (Corollary \ref{cor:genwienholtz}).  Consequently the length $L_1$ of the projection of $\gamma$ into the line  orthogonal to $H_i$ must be at least $3w$ for general curves, and $4w$ for closed curves. On the other hand we can also bound from below the length $L_2$ of the projection of $\gamma$ into $H_i$ by known results in $\R^2$. These estimates yield \eqref{eq:0} and \eqref{eq:1} due to the basic inequality $L\geq\sqrt{L_1^2+L_2^2}$  (Lemma \ref{lem:wienholtz}). 
 The proof of \eqref{eq:2} also follows from Wienholtz's theorem, once we utilize the theorem of Joris on the  inradius problem for planar curves. To prove \eqref{eq:3}, on the other hand, we develop the notion of \emph{horizon} of a curve, $H(\gamma)$, which is the measure of the tangent planes of $\S^2$, counted with multiplicity, that intersect $\gamma$. In Section \ref{sec:main2}, we derive upper and lower bounds for the horizon, which lead to the proof of \eqref{eq:3}. Finally in Section \ref{sec:general} we 
 discuss some generalizations of our estimates, including an
 extension of the inradius estimate to the notion of $n^{th}$ \emph{hull} in geometric knot theory \cite{CKKS}.
 
\section{Preliminaries: Projections of Length}

Here we will record some basic lemmas on length of projections of curves which will be useful throughout the paper.
In this work a \emph{curve} is a continuous mapping $\gamma\colon[a,b]\to\R^n$. By abuse of notation, we also use $\gamma$ to denote its image $\gamma([a,b])$, and  say that $\gamma$ is \emph{closed} if $\gamma(a)=\gamma(b)$. 
The \emph{length} of $\gamma$ is defined as
$$
L=L[\gamma]:=\sup \sum_{i=1}^n \|\gamma(t_i)-\gamma(t_{i-1})\|,
$$
where the supremum is taken over all partitions $a:=t_0\leq t_1\leq\dots\leq t_n:=b$ of $[a,b]$, $n\in\mathbf{N}$. We say that $\gamma$ is \emph{rectifiable} provided that $L$ is finite. Further $\gamma$ is \emph{parametrized by arclength} or has \emph{unit speed} if $L[\gamma|_{[t_1,t_2]}]=t_2-t_1$ for all $t_1$, $t_2\in [a,b]$. A curve $\tilde\gamma\colon [c,d]\to \R^n$ is a \emph{reparametrization} of $\gamma$ provided that there exists a nondecreasing continuous map $\phi\colon [a,b]\to[c,d]$ such that $\gamma=\tilde\gamma\circ\phi$. It is easy to see that $L[\gamma]=L[\tilde\gamma]$. If $\tilde\gamma$ has unit speed, then we say that it is a reparametrization of $\gamma$ by arclength. 

 The first lemma below collects some basic facts from measure theory, which allow us to extend some well-known analytic arguments  to all rectifiable curves.

\begin{lem}\label{lem:1}
Let $\gamma\colon[a,b]\to\R^n$ be a rectifiable curve.
\begin{enumerate}
\item{If $\gamma$ is Lipschitz, then $\gamma'(t)$ exists for almost all $t\in[a,b]$ (with respect to the Lebesgue measure), and 
$
L=\int_a^b \|\gamma'(t)\|\,dt.
$}
\item{If $\gamma$ is parametrized by arclength, then
$
\|\gamma'(t)\|=1,
$
for almost all $t\in[a,b]$.}

\item{There exists a (Lipschitz) reparametrization  of $\gamma$ by arclength.}
\end{enumerate}
\end{lem}
\begin{proof}
The first statement is just Rademacher's theorem. The second statement is proved in \cite[Thm. 2]{pelling}, see also \cite[Thm. 2.7.6]{bbi:book} or \cite[Thm. 4.1.6]{ambrosio}. For the third statement see \cite[2.5.16]{federer:book} or \cite[Prop. 2.5.9]{bbi:book}.
\end{proof}

The following lemma is a quick generalization of an observation of Wienholtz \cite{wienholtz:diameter}, which had been obtained by polygonal approximation,  see also \cite[Lem. 8.2]{sullivan:2008}. Here we offer a quick analytic proof which utilizes the above lemma.

\begin{lem}[Length Decomposition]\label{lem:wienholtz}
Let $\gamma\colon[a,b]\to\R^n$ be a rectifiable curve, $V_i$ be a collection of pairwise orthogonal subspaces   which span $\R^n$, and $L_i$ be the length of the orthogonal projection of $\gamma$ into $V_i$.
Then
$$
L\geq\sqrt{\sum_i L_i^2}.
$$ 
\end{lem}
\begin{proof}
By Lemma \ref{lem:1} we may assume that $\gamma$ is parametrized by arclength. Then  $\gamma'$ exists and $\|\gamma'\|=1$ almost everywhere.
Let $\gamma_i$ denote the orthogonal projection of $\gamma$ into $V_i$.  Then $\gamma_i$ are also Lipschitz, so $\gamma_i'$ exist almost everywhere as well. Since $V_i$ are orthogonal and span $\R^n$, $\gamma=\sum_i\gamma_i$, which yields $\gamma'=\sum_i\gamma_i'$. Further,  $\lan\gamma'_i,\gamma'_j\ran=0$ for $i\neq j$, since $\gamma_i'\in V_i$. So 
$$
\sum_i\|\gamma_i'\|^2=\|\gamma'\|^2=1.
$$
Now the  Cauchy-Schwartz inequality  yields
$$
\sum_i(L_i)^2= \sum_i\(\int_a^b\|\gamma'_i\|\)^2
\leq (b-a)\sum_i\int_a^b\|\gamma'_i\|^2  \\
= (b-a)^2=L^2,
$$
which completes the proof.
\end{proof}

The next observation we need is a general version of a classical result which goes back to Cauchy. Originally this result was proved for smooth curves; however, it is well-known that it holds for all rectifiable curves  \cite{ayari-dubuc}. Here we simply record that the original analytic proof may be extended almost verbatim to the general case in light of Lemma \ref{lem:1}.

\begin{lem}[Cauchy-Crofton]\label{lem:cc}
Let $\gamma\colon[a,b]\to\R^2$ be a rectifiable curve, $u\in\S^1$, and $\gamma_u$ be the projection of $\gamma$ into the line spanned by $u$. Then
$$
L=\frac{1}{4}\int_{\S^1} L[\gamma_u]\,du.
$$
\end{lem}
\begin{proof}
Again, by Lemma \ref{lem:1}, we may assume that $\gamma$ is Lipschitz (after a reparametrization by arclength), which, since $\gamma_u=\langle \gamma,u\rangle u$, yields that $\gamma_u$ is Lipschitz as well. Thus
\begin{eqnarray*}
\int_{\S^1} L[\gamma_u]\,du &=& \int_{\S^1}\int_a^b \|\gamma_u'(t)\|\,dt = \int_{\S^1}\int_a^b |\lan\gamma'(t),u\ran|\,dt\\
&=& \int_{0}^{2\pi}\int_a^b \|\gamma'(t)\|\, |\cos(\theta(t))|\,dt d\theta\\
&=& 4 L,
\end{eqnarray*}
as desired.
\end{proof}

The last lemma immediately yields another classical fact, which had been originally proved for smooth curves:

\begin{cor}[Barbier]\label{cor:cauchy}
Let $\gamma\colon[a,b]\to\R^2$ be a closed rectifiable curve of width $w$. Then
$$
L\geq\pi w.
$$
\end{cor}
\begin{proof}
Let $\gamma_u$ be as in Lemma \ref{lem:cc}. Then
$$
L=\frac{1}{4}\int_{\S^1} L[\gamma_u]\geq \frac{1}{4}\int_0^{2\pi} 2 w\,d\theta=\pi w.
$$
\end{proof}

\section{The Theorem of Wienholtz}\label{sec:wienholtz}

Motivated by a 1998 conjecture of Kusner and Sullivan \cite{kusner&sullivan:distortion} on diameter of space curves, Wienholtz made  the following fundamental observation in an unpublished work in 2000 \cite{wiehnoltz:parallel}; see also \cite[Sec. 8]{sullivan:2008} where the argument is outlined. Here we include a complete proof which is significantly shorter than the original, although it is  based on the same essential idea. 

\begin{thm}[Wienholtz]\label{thm:wienholtz}
For any continuous map $\gamma\colon \S^1\to\R^n$  there exists a pair of parallel hyperplanes $H_0$, $H_1\subset\R^n$, and $4$ points $t_0$, $t_1$, $t_2$,  $t_3$ cyclically arranged in $\S^1$ such that $\gamma$ lies in between $H_0$, $H_1$, while $\gamma(t_0)$, $\gamma(t_2)\in H_0$, and $\gamma(t_1)$, $\gamma(t_3)\in H_1$.
\end{thm}
\begin{proof}
For every direction $u\in\S^{n-1}$, let  $H_u$ be the support hyperplane of  $\gamma$ with outward normal $u$, and set $X_u:=\gamma^{-1}(H_u)$.  We may assume that $H_u\neq H_{-u}$  for all $u$, for otherwise there is nothing to prove. Let $S_u$ be the open slab bounded by $H_{\pm u}$, and $I_u\subset \S^1$ be a connected component of $\gamma^{-1}(S_u)$ such that the initial point of $I_u$ (with respect to some fixed orientation of $\S^1$) lies in $X_{-u}$ and its final end point lies in  $X_u$. 
If $I_u$ is not unique for some $u$,  then we are done. So suppose, towards a contradiction, that $I_u$ is unique for every $u$. Then we will construct a map $f\colon\S^{n-1}\to\R^2\subset\R^{n-1}$ such that $f(u)\neq f(-u)$ for all $u$.
This violates the Borsuk Ulam theorem, and completes the proof.

To construct $f$, 
pick a  point $p(u)\in I_u$ for each $u$. We claim that there exists an open neighborhood $V(u)$ of $u$ in $\S^{n-1}$ such that
\begin{equation}\label{eq:mu}
p(u)\in I_{u'}\text{    for all   }u'\in V(u).
\end{equation} 
Indeed, since $I_u$ is unique, the compact sets $X_u$ and $X_{-u}$ lie  in the interiors of the (oriented) segments $p(u)p(-u)$ and $p(-u)p(u)$ of $\S^1$ respectively (Figure \ref{fig:circle}(a)). Thus, since $u\mapsto H_u$ is continuous, we may choose $V(u)$ so small that $X_{u'}$ lies  in the  interior of  $p(u)p(-u)$ and $X_{-u'}$ lies  in the in interior of  $p(-u)p(u)$. Let $J_{u'}$ be the component of $\gamma^{-1}(A_{u'})$ which contains $p(u)$. Then the final boundary point of $J_{u'}$ must be in $X_{u'}$ and its initial boundary point   must be in $X_{-u'}$. Thus $J_{u'}=I_{u'}$ by the uniqueness property, which establishes \eqref{eq:mu}.

   \begin{figure}[h]
   \centering
    \begin{overpic}[height=1.25in]{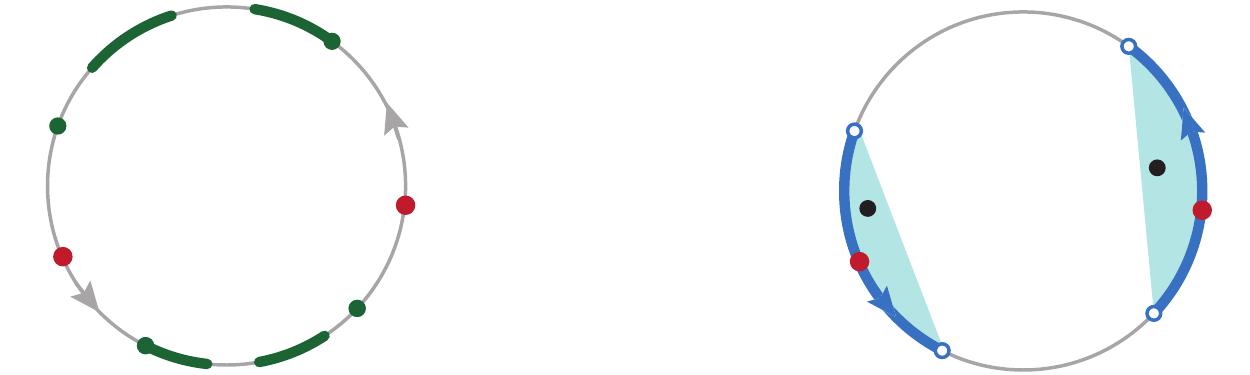} 
    \put(34,13){\small$p(u)$}
    \put(-4.5,8.25){\small$p(-u)$}
    \put(6,28.5){\small$X_u$}
    \put(24,-1){\small$X_{-u}$}
    % \put(98,13){\small$p(u)$}
    %\put(58.75,8){\small$p(-u)$}
    \put(97,20){\small$I_u$}
    \put(61.5,12){\small$I_{-u}$}
     \put(85,16){\small$f(u)$}
      \put(70,12){\small$f(-u)$}
       \put(16.5,-5){\small$(a)$}
        \put(81,-5){\small$(b)$}
    \end{overpic}
   \caption{}
   \label{fig:circle}
\end{figure}

Let $u_i\in\S^{n-1}$ be a finite number of directions   such that  $V(u_i)$ cover $\S^{n-1}$, 
$\phi_i\colon\S^{n-1}\to\R$ be a partition of unity subordinate to $\{V(u_i)\}$, and set
$$
f(u):=\sum_i \phi_i(u) p(u_i).
$$
If $\phi_i(u)\neq 0$, for some $i$, then $u\in V(u_i)$, and so \eqref{eq:mu} yields that
 $
 p(u_i)\in I_{u}.
 $ 
Then, 
 $
 f(u)\in\conv (I_u),
 $
 the smallest convex set containing $I_u$. But $\conv(I_u)\cap\conv(I_{-u})=\emptyset$,  because $I_u\cap I_{-u}=\emptyset$ (Figure \ref{fig:circle}(b)).
Thus $f(u)\neq f(-u)$  as claimed.
\end{proof}

Although Wienholtz stated his theorem only for closed curves, it is easy to see that it holds for all curves:

 \begin{cor}\label{cor:genwienholtz}
 For any curve $\gamma\colon[a,b]\to\R^n$ there exist four points $t_0<t_1<t_2<t_3$ in $[a,b]$ and a pair of parallel hyperplanes $H_0$, $H_1$ in $\R^n$ such that $\gamma(t_0)$, $\gamma(t_2)\in H_0$ and  $\gamma(t_1)$, $\gamma(t_3)\in H_1$.
 \end{cor}
 \begin{proof}
 If $\gamma(a)=\gamma(b)$, we may identify $[a,b]$ with $\S^1$ and we are done by Theorem \ref{thm:wienholtz}. Further note that in this case we may assume that $t_i\in[a,b)$. If $\gamma(a)\neq\gamma(b)$,  let $\ell$ be the line segment connecting $\gamma(a)$ and $\gamma(b)$. Then we may extend $\gamma$ to a closed curve $\tilde\gamma\colon [a,b']\to\R^n$, for some $b'>b$ such that $\tilde\gamma|_{[b,b']}$ traces $\ell$. By Theorem \ref{thm:wienholtz}, there are points $t_0<t_1<t_2<t_3$ in $[a,b')$ such that  $\gamma(t_0)$, $\gamma(t_2)\in H_0$ and  $\gamma(t_1)$, $\gamma(t_3)\in H_1$. If interior of $\ell$ is disjoint from $H_0$ and $H_1$ then $t_i\not\in (b,b')$ and we are done. If interior of $\ell$ intersects $H_j$, then $\ell$ lies entirely in $H_j$. In this case suppose that $t_i\in[b,b']$. Then $\gamma([t_i,b'])$ lies in $H_j$. So it would follow that $i=3$, for otherwise $\gamma(t_i)$ and $\gamma(t_{i+1})$ would lie in the same hyperplane which is not possible. Now that $\gamma(t_3)$ lies in $H_j$, it  follows that $t_2<b$; because $\gamma([b,t_3])$ lies in $H_j$ and $\gamma(t_2)$ and $\gamma(t_3)$ cannot lie in the same hyperplane. Thus we may replace 
  $t_3$ with $b$ which concludes the proof.
 \end{proof}

\section{Estimates for Width: Proof of Theorem \ref{thm:main1}}

Using the generalized Wienholtz Theorem (Corollary \ref{cor:genwienholtz}) together with the length decomposition lemma (Lemma \ref{lem:wienholtz}), we will now  prove our main inequalities for the width:

\begin{proof}[Proof of Theorem \ref{thm:main1}]
Let $H_0$, $H_1$ be a pair of parallel bounding planes of $\gamma$ as in the generalization of Wienholtz's theorem, Corollary \ref{cor:genwienholtz}. Further let   $\gamma_1$ be the projection of $\gamma$ into a line orthogonal to $H_0$, $\gamma_2$ be the projection of $\gamma$ into $H_0$, and $L_i$, $w_i$ denote the length and width of $\gamma_i$ respectively.Then 
$$
L_1\geq L[\gamma_1|_{[t_0,t_1]}]+L[\gamma_1|_{[t_1,t_2]}]+L[\gamma_1|_{[t_2,t_3]}]\geq3w.
$$
 Further recall that, for planar curves,  $L/w$ is minimized for Zalgaller's caliper curve, where this quantity is bigger than  $2.2782$ \cite{adhikari&pitman}. Thus
 $$
L_2\geq  2.2782 \,w_2 \geq  2.2782 \,w,
 $$
So, by Lemma \ref{lem:wienholtz}, 
$$
L\geq\sqrt{L_1^2+L_2^2}\geq\sqrt{(3w)^2+(2.2782 w)^2}\geq 3.7669\, w,
$$
which establishes \eqref{eq:0}. Next, to prove \eqref{eq:1}, suppose that $\gamma$ is closed. 
Then 
$
L_1\geq  4w.
$
 Further, it follows from Corollary \ref{cor:cauchy} that $L_2\geq \pi w_2\geq \pi w$. Thus, again by Lemma \ref{lem:wienholtz}, 
$$
L\geq\sqrt{L_1^2+L_2^2}\geq\sqrt{(4w)^2+(\pi w)^2},
$$
which completes the proof.
\end{proof}

\begin{note}[The case of equality in \eqref{eq:1}]
As we discussed above, when $\gamma$ is closed, $L_1\geq  4w$ and $L_2\geq \pi w$. Thus the last displayed expression in the above proof shows that equality holds in \eqref{eq:1} only if
$L_1=  4w$ and $L_2= \pi w$. It follows then that $\gamma_2$ is a curve of constant width, and $\gamma$ is composed of four geodesic segments in the cylindrical surface over $\gamma_2$. Since optimal objects in nature are usually symmetric, and $\gamma$ has four segments running between $H_0$ and $H_1$, it would be reasonable to expect that the minimal curve $\gamma$ would be symmetric with respect to a pair of orthogonal planes parallel to $u$. This would in turn imply that $\gamma_2$ is centrally symmetric. The only centrally symmetric curves of constant width, however, are circles. Thus if the equality in \eqref{eq:1} is achieved by a symmetric curve, then $\gamma_2$ should be a circle. As we show in the next section, however, the equality in \eqref{eq:1} never holds for curves which project onto a circles. Thus either \eqref{eq:1} is not quite sharp or else the minimal curve is not so symmetric. On the other hand, to add to the mystery, we will produce a symmetric curve in the next section which very nearly realizes the case of equality in \eqref{eq:1}.
\end{note}

\section{A Near Minimizer for $L/w$}\label{sec:example}
Here we construct a symmetric piecewise geodesic closed curve on a circular cylinder which shows that \eqref{eq:1} is very nearly sharp. To this end let
 $C_h$ be the cylinder of radius $1$ and height $h$  in $\R^3$ given by 
$$
x^2+y^2=1,\quad\quad -h/2\leq z\leq h/2.
$$
 Consider the $4$ consecutive points 
 $$
 p_1=(1,0,h/2), \;p_2=(0,1,-h/2), \;p_3=(-1,0,h/2), \;p_4=(0,-1,-h/2)
 $$ 
 on the boundary of $C_h$. Let $\Gamma_h$ be the simple closed curve consisting of $4$ geodesic or helical segments connecting these points cyclically, see Figure \ref{fig:cylinder}. Note that, as the figure shows, $\Gamma_h$ may also be constructed by rolling a planar polygonal curve onto the cylinder. Let $L(h)$ and $w(h)$ denote the length and width of $\Gamma_h$ respectively. We will show that $L(h)/w(h)$ is minimized when the height of $C_h$ is slightly smaller than its diameter, or more precisely $h=h_0\approx 1.97079$. Then 
 $L(h_0)/w(h_0)< 5.1151$, and thus $\Gamma_{h_0}$ yields the curve which we mentioned in the introduction.
 
 \begin{figure}[h]
   \centering
    \begin{overpic}[height=1.25in]{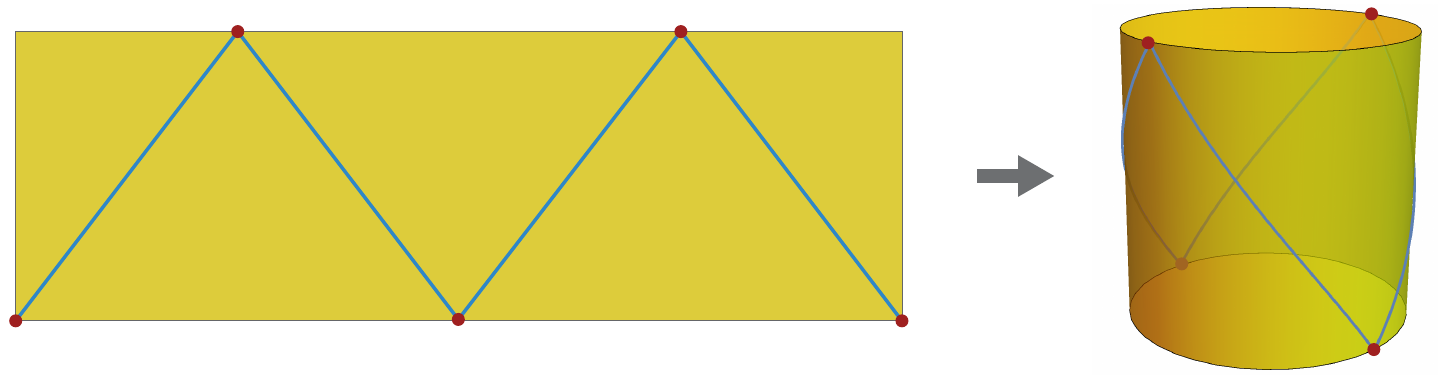} 
     \put(30.5,1){\small $2\pi$}
     \put(-2,13){\small $h$}
     \put(100,13){\small $\Gamma_h$}
     \put(95,26.5){\small $p_1$}
      \put(82,5.5){\small $p_2$}
       \put(80,24.5){\small $p_3$}
       \put(95,-0.5){\small $p_4$}
    \end{overpic}
   \caption{}
   \label{fig:cylinder}
\end{figure}

To compute $h_0$,  first note that, by the Pythagorean theorem,
$$
L(h)=\sqrt{(4h)^2+(2\pi)^2}.
$$
Next, to find $w(h)$, let $\ol\Gamma_h$ be the projection of $\Gamma_h$ into the $xz$-plane, see Figure \ref{fig:graphs}, and let $\ol w(h)$ denote the width of $\ol\Gamma_h$, i.e., the infimum of the distance between all pairs of parallel lines in the plane of $\ol\Gamma_h$ which contain $\ol\Gamma_h$. Now we record the following lemma, whose proof we will postpone to the end of this section.

\begin{lem}\label{lem:widthprojection}
The width of $\Gamma_h$ is equal to the width of $\ol\Gamma_h$:
$$
w(h)=\ol w(h).
$$
\end{lem}

To compute $\ol w(h)$, note that there are two possibilities: (i) $\ol w(h)=h$, or (ii) $\ol w(h)$ is given by the distance $d(h)$ between one of the end points of $\ol\Gamma_h$ and the opposite branch of $\ol\Gamma_h$. So we have
$$
\ol w(h)=\min\{h, d(h)\}.
$$

\begin{figure}[h]
   \centering
    \begin{overpic}[height=1.25in]{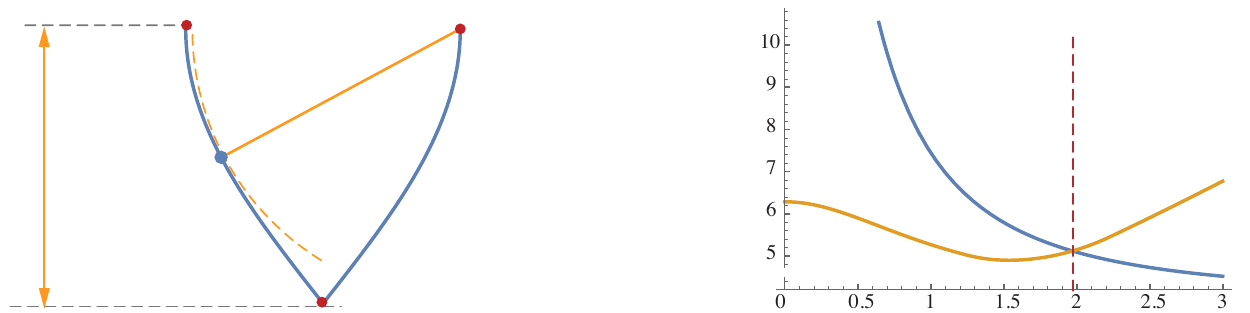} 
    \put(71,23){\small $\frac{L(h)}{h}$}
    \put(98,11){\small $\frac{L(h)}{d(h)}$}
    \put(100,0.5){\Small $h$}
     \put(36,11){\small $\ol\Gamma_h$}
     \put(1,11){\small $h$}
     \put(22.5,19){\small $d(h)$}
      \put(36,24.5){\small $p_1$}
      \put(14,24.5){\small $p_3$}
     \put(86.5,21){\small $h=h_0$}
    \end{overpic}
   \caption{}
   \label{fig:graphs}
\end{figure}

To find $d(h)$ note that the left branch of $\ol\Gamma_h$ may be parametrized by 
$$
(-\cos(t), 0, (1/2+{2t}/{\pi})h),\quad\quad-\pi/2\leq t\leq 0,
$$ and the tip of the right branch is $p_1=(1,0,h/2)$. Thus
$$
d(h)=\min_{-\pi/2\leq t\leq 0}\sqrt{(\cos(t)+1)^2+((2t/\pi)h)^2}.
$$
Finally we have
$$
\frac{L(h)}{w(h)}=\frac{L(h)}{\ol w(h)}=\max\left\{\frac{L(h)}{h}, \frac{L(h)}{d(h)}\right\}.
$$
Graphing these functions shows that $L(h)/w(h)$ is minimized when $\frac{L(h)}{h}=\frac{L(h)}{d(h)}$ or $h=d(h)$; see  Figure \ref{fig:graphs}.

Now let $h_0$ denote the solution to $h=d(h)$. Via a computer algebra system, we can find that $1.97078<h_0<1.97080$. Consequently
$$
\min\(\frac{L(h)}{w(h)}\)=\frac{L(h_0)}{w(h_0)} =\frac{\sqrt{(4h_0)^2+(2\pi)^2}}{h_0}<\frac{\sqrt{(4\times1.97080)^2+(2\pi)^2}}{1.97078}< 5.1151.
$$
as desired.

It only remains now to prove the last lemma. To this end we need to consider the boundary structure of the convex hull $C$ of $\Gamma_h$. Note that $\Gamma_h$ lies on the boundary $\partial C$ of $C$, and divides $\partial C$ into a pair of regions by the Jordan curve theorem; see Figure \ref{fig:ch}. Further, each of these regions is a ruled surface. In one of the regions all the rulings are parallel to $p_1p_3$, or the $x$-axis, while in the other region the rulings are parallel to $p_2p_4$ or the $y$-axis. Each of these regions can be subdivided into a pair of triangular regions by the lines $p_1p_3$ and $p_2p_4$. Thus we may say that $\partial C$ carries a tetrahedral structure, and call these subregions the faces of $\partial C$. For instance, the face of $\partial C$ which is opposite to $p_1$ is $p_2p_3p_4$.

\begin{proof}[Proof of Lemma \ref{lem:widthprojection}]
By definition of width, we have $w(h)\leq\ol w(h)$. So we just need to establish the reverse inequality. To this end, let $H$, $H'$ be a pair of parallel planes, with separation distance $w(h)$, which contain $\Gamma_h$ in between them. First suppose that  $H$ (or $H'$) intersects the convex hull $C$ of $\Gamma_h$ at more than one point. Then, since $C$ is convex, and $H$ is a support plane of $C$, $H$ must contain a line segment in $\partial C$, the boundary of $C$. All line segments in $\partial C$ are parallel to either the $x$-axis or the $y$-axis, as we discussed above. Thus $H$, and consequently $H'$ must be parallel to, say, the $y$-axis. Consequently, if we let $\ell$, $\ell'$ be the intersections of $H$, $H'$ with the $xz$-plane, it follows that $\ol\Gamma_h$ is contained between $\ell$ and $\ell'$, which are separated by the distance $w(h)$. Thus $\ol w(h)\leq w(h)$, as desired.

\begin{figure}[h]
   \centering
    \begin{overpic}[height=1.55in]{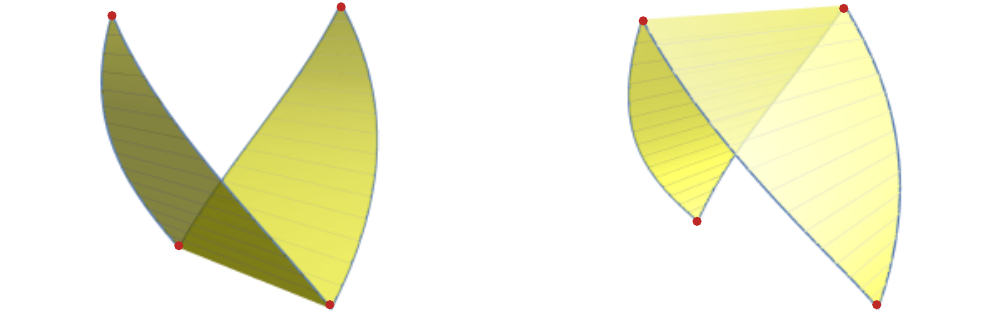} 
      \put(10,31){$p_3$}
      \put(33,32){$p_1$}
      \put(16,5){$p_2$}
      \put(33,-1){$p_4$}
      \put(63,31){$p_3$}
        \put(85,32){$p_1$}
      \put(69,7){$p_2$}
      \put(86,-1){$p_4$}
     \end{overpic}
   \caption{}
   \label{fig:ch}
\end{figure}

We may suppose then that $H$ and $H'$ intersect $C$ at precisely one point each, which we call $p$ and $p'$ respectively. Then the line segment $pp'$ must be orthogonal to $H$ and $H'$ (e.g., see \cite[p. 86]{schneider:book}). Further, if $H\cap C$ and $H'\cap C$ are singletons, then $H$ and $H'$ can intersect $C$ only along $\Gamma_h$, because $\partial C-\Gamma_h$ is fibrated by line segments. Now suppose that both  $p$ and $p'$ belong to the interior of branches of $\Gamma_h$, i.e., the complement of $p_i$. Then  $pp'$ must be orthogonal to $\Gamma_h$ at both ends. This may happen only when $p$ and $p'$ belong to a pair of opposite branches of $\Gamma_h$, and $pp'$ is parallel to the $xy$ plane. It follows then that $\|pp'\|=2$, which yields $w(h)=2$. On the other hand $\ol w(h)\leq 2$, since the distance between the end points of $\ol\Gamma_h$ is $2$. Thus again we obtain $\ol w(h)\leq w(h)$.

It only remains then to consider the case where $p$ is an end point of a branch of $\Gamma_h$, say $p=p_1$. In this case $p'$  must belong to one of the branches of $\Gamma_h$ which is not adjacent to $p_1$, i.e., either  $p_2p_3$ or $p_3p_4$. So $p'$ must belong to face or the triangular region $p_2p_3p_4$ in $\partial C$. Consequently
$$
w(h)\geq \dist(p_1, p_2p_3p_4)=d(h)\geq\ol w(h),
$$
which completes the proof. Here $d(h)$ is the distance between $p_1$ and the opposite branch of $\ol\Gamma_h$, as we had discussed above.
\end{proof}

\begin{note}
Although the curve $\Gamma_{h_0}$ we constructed above is the minimizer for the width problem among  curves on a circular cylinder, it is not
the minimizer for the width problem among all closed curves. Indeed we may replace small segments of $\Gamma_{h_0}$ which have an end point at $p_i$ with straight line segments without decreasing the width. 
\end{note}

\section{Zalgaller's $L_5$ Curve}
In \cite{zalgaller:1994} Zalgaller describes a closed space curve, ``$L_5$", which he claims minimizes the ratio $L/w$. Here we show that this conjecture is not true. Indeed, the ratio $L/w$ for Zalgaller's curve, which here we call $Z$,  is bigger than that of the cylindrical curve $\Gamma_{h_0}$ which we constructed in Section \ref{sec:example}. Since Zalgaller does not calculate $L/w$ for this example, we include this calculation below. We will begin by describing the construction of $Z$, since Zalgaller's paper is not available in English.

The curve $Z$ is modeled on a regular tetrahedron. Note that the width of a regular tetrahedron is the distance between any pairs of its opposite edges. In particular this distance is $1$ when the side lengths are $\sqrt 2$. The  basic idea  for constructing $Z$ is to take a simple closed curve, which traces $4$ consecutive edges of a tetrahedron, say of edge length $\sqrt{2}$, and reduce its length without reducing its width. The error in Zalgaller's construction, however, is that the width does go down below $1$, as we will show below. 

\subsection{Construction}

\begin{figure}[h]
   \centering
    \begin{overpic}[height=1.75in]{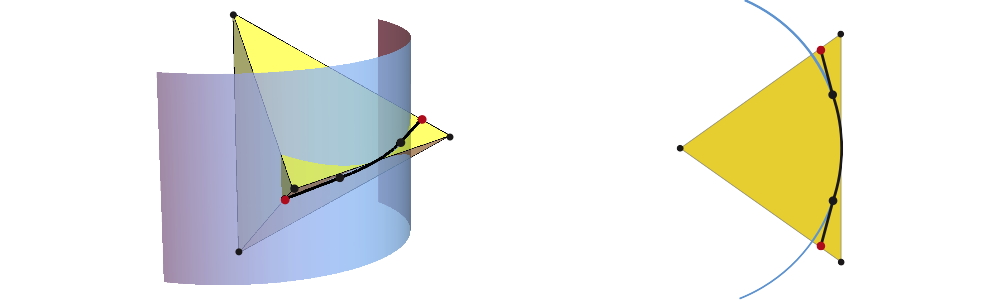} 
      \put(22,29.5){\Small $C$}
      \put(22.5,2.5){\Small $D$}
      \put(28.75,12.5){\Small $B$}
      \put(45.5,15.5){\Small $A$}
      \put(42,19.5){\Small $A'$}
      \put(27,7.5){\Small $B'$}
       \put(33.5,10){\Small $F$}
       \put(38.5,16.5){\Small $E$}
      \put(65.25,15.75){\Small $\ol C$}
      \put(65,13.25){\Small $\ol D$}
      \put(84.5,27.5){\Small $\ol A$}
      \put(84,1.25){\Small $\ol B$}
       \put(80.5,26){\Small $\ol{A'}$}
       \put(80,2.75){\Small $\ol{B'}$}
       \put(84.5,20){\Small $\ol E$}
       \put(84.5,9){\Small $\ol F$}
        \end{overpic}
   \caption{}
   \label{fig:zalgaller}
\end{figure}

Take a regular tetrahedron $T$ with vertices $A$, $B$, $C$, $D$, as shown in Figure \ref{fig:zalgaller}. Assume that the edge lengths are $\sqrt 2$ so that the width of $T$ is $1$, i.e., the distance between the edges $AB$ and $CD$. Let $A'$ be the point on $AC$ whose distance from the face $BCD$ is $1$. A simple computation shows that $A'$ is the point on $AC$ whose distance from $C$ is $\sqrt 6/2$. Similarly, let $B'$ be the point on $BD$ whose distance from $D$ is $\sqrt 6/2$. Let $X$ be the cylinder of radius $1$ with axis $CD$. Now connect $A'$ and $B'$ with the shortest arc  which lies outside $X$. Note that this arc is composed of a pair of straight line segments, plus a helical segment which lies on $X$. This forms the side $A'B'$ of $Z$. Similarly, we can form the sides $B'C'$, $C'D'$ and $D'A'$ which will yield the whole curve as shown in Figure \ref{fig:tetrahedra}.

\subsection{Length}\label{subsec:lengthL5}
To compute the length of $Z$, we are going to assume that $A=(1,\sqrt2/2, 0)$, $B=(1, -\sqrt2/2,0)$, $C=(0,0, \sqrt2/2)$, and $D=(0,0,-\sqrt2/2)$. For any set $S\subset\R^3$, let $\ol S$ denote its projection into the $xy$-plane, and note that $\ol C$ and $\ol D$  coincide with the origin $o$ of the $xy$-plane.  So the cylinder $X$ will intersect the $xy$-plane in a circle of radius $1$ centered at $o$; see the right diagram in Figure \ref{fig:zalgaller}. Let $\ol{A'B'}$ denote the projection of $A'B'$ into the $xy$-plane, and $h$ be the distance of $A'$ or $B'$ from the $xy$-plane. Then
$$
L(A'B')=\sqrt{L(\ol{A'B'})^2+(2h)^2}.
$$
The reason behind the above equality is that $A'B'$ lies on the cylindrical surface over $\ol{A'B'}$, and is a geodesic in that surface (which has zero curvature); thus, the Pythagorean theorem applies. Next,  note that
$$
A'=\(1-\frac{\sqrt3}{2}\)C+\frac{\sqrt 3}{2}A, =\(\frac{\sqrt3}{2},\frac{\sqrt3}{2\sqrt2},\frac{2-\sqrt3}{2\sqrt2}\).
$$
Thus
$$
\ol{A'}=\(\frac{\sqrt3}{2},\frac{\sqrt3}{2\sqrt2}\),\quad\text{and}\quad h=\frac{2-\sqrt3}{2\sqrt2}.
$$
Next, to compute $L(\ol{A'B'})$, note that $\ol{A'B'}=\ol{A'E}\cup \ol{EF} \cup \ol{FB'}$, where $A'E$ and $FB'$ are line segment, and $EF$ is a circular arc. To find $L(\ol{EF})$, write $\ol E=(\cos(\theta),\sin(\theta))$, and note that   $\langle \ol E-\ol A',\ol E\rangle=0$, which yields that $\theta=\tan^{-1}(\sqrt 5/2)$. So $L(\ol{EF})=2\arctan(\sqrt 5/2)$. Further, it follows that $\ol E=(5,\sqrt2)/(3\sqrt3)$, which in turn allows us to compute that $L(\ol{A'E})=\sqrt2/4$. So we conclude that
$$
L(\ol{A'B'})=2L(\ol{A'E})+L(\ol{EF})=\frac{1}{\sqrt2}+2\tan^{-1}\(\frac{\sqrt5}{2}\),
$$
which in turn yields
$$
L(Z)=4L(A'B')=4\sqrt{\left(\frac{1}{\sqrt{2}}+2 \tan
   ^{-1}\left(\frac{\sqrt{2}}{5}\right)\right)^2+\left(\frac{2-\sqrt{3}}{\sqrt2}\right)^2}\approx 5.0903
$$

\begin{figure}[h]
   \centering
    \begin{overpic}[height=1.25in]{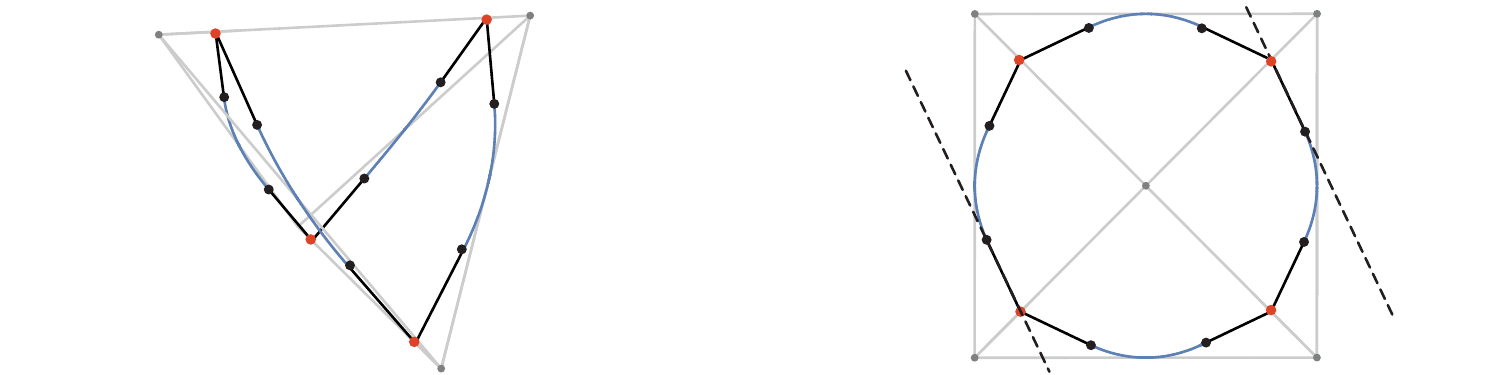} 
      \put(35.75,23.5){\Small $A$}
      \put(34,17.5){\Small $E$}
       \put(14,11){\Small $E'$}
      \put(28.5,-2){\Small $B$}
      \put(8,22){\Small $C$}
      \put(17,7.5){\Small $D$}
      \put(62,23.5){\Small $\tilde D$}
       \put(62,7.5){\Small $\tilde E'$}
      \put(88.25,23.5){\Small $\tilde A$}
       \put(88.25,15.75){\Small $\tilde E$}
        \put(75,9){\Small $\tilde M$}
      \put(62,0){\Small $\tilde C$}
      \put(88,0){\Small $\tilde B$}
        \end{overpic}
   \caption{}
   \label{fig:tetrahedra}
\end{figure}

\subsection{Width}
To estimate the width of $Z$ we are going to project it into a plane $\Pi$ orthogonal to  
$$
u:=\frac{A+C}{2}-\frac{B+D}{2}=\(0,\frac{1}{\sqrt{2}}, \frac{1}{\sqrt{2}}\);
$$ see the right diagram in Figure \ref{fig:tetrahedra}. For any set $S\subset\R^3$, we let $\tilde S$ denote its projection into $\Pi$. Let $E'$ be the point on the segment $DC$ of $Z$ which lies at the end of the line segment in $DC$ starting at $D$. Further let $M$ denote the center of mass of the tetrahedron $T$. Note that
$$
w(Z)\leq w(\tilde Z)\leq \|\tilde{EE'}\|=2\|\tilde{EM}\|.
$$
 The first inequality above is obvious from the definition of $w$; the second inequality follows from the fact that $\tilde Z$ is contained in between the lines spanned by $\tilde{A'E}$ and $\tilde{C'E'}$; and the last equality of course is due to the fact that $\tilde M$ is the midpoint of $\tilde{EE'}$. It only remains then to compute $\|\tilde{EM}\|$. 
 To this end first note that 
 $$
 M=(A+B+C+D)/4=(1/2,0,0).
 $$
  Next, to compute $E$, recall that we already computed its first two components given by $\ol E=(5,\sqrt2)/(3\sqrt3)$. To find the third component of $E$ recall that $A'B'$ is a linear graph over its projection $\ol{A'B'}$. More specifically, the height of this graph is given by $z(t)=\frac{h}{L/8}\;t$, where $t$ measures the distance from the center of $\ol{A'B'}$. Thus
 $$
 E=\(\frac{5}{3\sqrt3},\frac{\sqrt2}{3\sqrt3}, \frac{h}{L/8}\tan^{-1}(\frac{\sqrt5}{2})\).
 $$
 Finally recall that $\tilde E=E-\langle E,u\rangle u$ and $\tilde M=M-\lan M,u\ran u$. So we now have all the information to compute that
 $$
w(Z)\leq 2 \|\tilde{EM}\|=2\|\tilde E-\tilde M\|\approx 0.980582.
 $$
 In particular, the computation contradicts Zalgaller's conjecture that $Z$ has width $1$. Using these computations, we now have
 $$
 \frac{L(Z)}{w(Z)}\geq 5.1911,
 $$
 which is bigger than $5.1151$, the ratio $L/w$ for the cylindrical curve we constructed in the last section. Thus $Z$ does not minimize $L/w$, contrary to Zalgaller's conjecture.

 \begin{note}[Original Statement of the $L_5$ Conjecture]\label{note:L5}
 Since Zalgaller's paper \cite{zalgaller:1994} is not available in English, here we include a translation of the conjecture on the shortest closed curve of width $1$, which we disproved above.

``16. A similar problem can be posed for closed curves. On the plane
any curve of constant width $1$ is the shortest closed curve of width $1$.
In space consider the regular tetrahedron with edge $\sqrt{2}$ (fig \ref{fig:zalgallerorig}).
\begin{comment}
\begin{wrapfigure}{r}{0.3\textwidth}
    \includegraphics[width=0.28\textwidth]{zalgallerorig.pdf}
     \caption{}
   \label{fig:zalgallerorig}
\end{wrapfigure}
\end{comment}
The $4$-segment polygonal curve $L_4=ABCDA$ is an example of a closed curve of width $1$.
 Mark the middle points $O_1$, $O_2$, $O_3$, $O_4$
on the edges of $L_4$. On the edge $AC$
which is not in $L_4$ mark the point $A^\prime$ which is at distance $1$ from
the plane $BCD$ and also mark a point $C^\prime$ which is at distance $1$
from the plane $ABD$. Similarly, on the edge $BD$ which is also not in 
$L_4$ mark points $B^\prime$, $D^\prime$ which are at distance $1$ 
from the planes $ACD$, $ABC$, respectively.
\begin{figure}[h]
 \centering
 \begin{overpic}[height=1.25in]{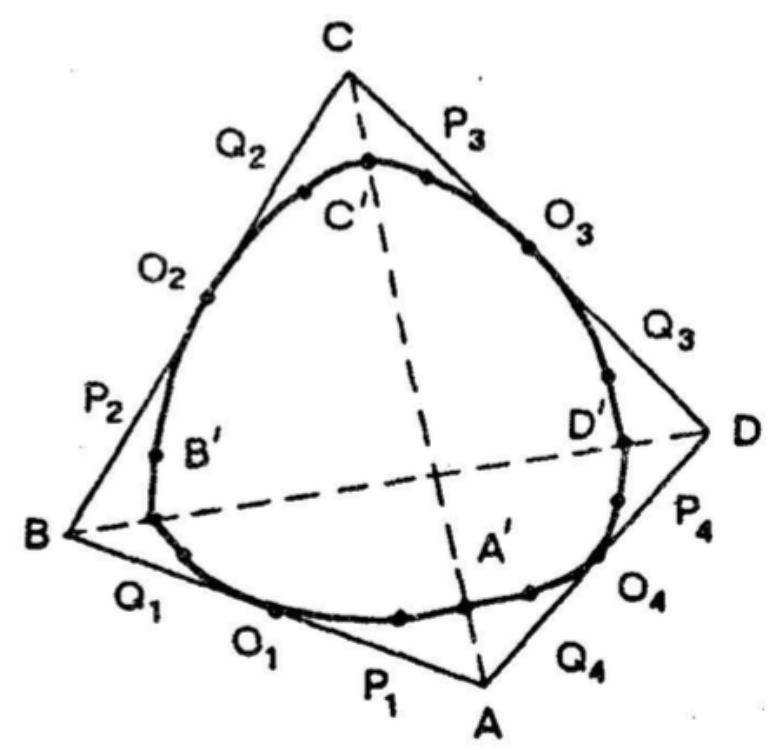} 
          \end{overpic}
  \caption{}
  \label{fig:zalgallerorig}
\end{figure}
Form the curve $L_5$ from four congruent  $C^1$-smooth portions
$A^\prime B^\prime$, $C^\prime D^\prime$, $D^\prime A^\prime$.
It is enough to describe the portion $A^\prime B^\prime$.
We construct it
as the shortest curve $A^\prime P_1 O_1 Q_1 B^\prime$ that joins 
$A^\prime$ and $B^\prime$ and goes around outside
the circular cylinder $Z$ of radius $1$ with axis $CD$.
This portion is of the form 
$A^\prime B^\prime=A^\prime P_1 + P_1 O_1 Q_1 + Q_1 B^\prime$
where $A^\prime P_1$ and $Q_1 B^\prime$ are straight line
segments and $P_1 O_1 Q_1$ is a screw-rotational arc on the cylinder $Z$.
Similarly, one constructs the portions 
$B^\prime C^\prime=B^\prime P_2 O_2 Q_2 C^\prime$,
$C^\prime D^\prime=C^\prime P_3 O_3 Q_3 D^\prime$,
$D^\prime A^\prime=D^\prime P_4 O_4 Q_4 A^\prime$.

17. Conjecture 2. The curve $L_5$ has width $1$ and is the shortest
closed space curve of width $1$."
\end{note}

\begin{note}
 If Zalgaller had been correct in his conjecture that the width of $L_5$ is $1$, then the ratio $L/w$ for this curve would have been approximately $5.0903$ according to our computation of length  in Section \ref{subsec:lengthL5}. In this sense, the $L_5$ conjecture predicted that $L/w\geq 5.0903$, which interestingly enough is within $0.01$\% of the lower bound \eqref{eq:1} in Theorem \ref{thm:main1}. 
 \end{note}

 \section{Estimates for Inradius: Proof of Theorem \ref{thm:main2}} \label{sec:main2}
 
  The general estimate in Theorem \ref{thm:main2} follows quickly from the Wienholtz theorem as was the case for the estimates for the width in the proof of Theorem \ref{thm:main1}. For the case of closed curves, however, we will work harder to obtain a better estimate via the notion of horizon developed below. 
   
 \subsection{The general case}
 Here we prove \eqref{eq:2}.
 Let $\gamma_i$ be as in the proof of Theorem \ref{thm:main1}, and $L_i$, $r_i$ denote the length, and inradius of the convex hull of $\gamma_i$ respectively. Then 
 $$
 L_1\geq 3w\geq 6r.
 $$ 
 Further note that $r_2$ is not smaller than the inradius of the convex hull of $\gamma$, which in turn is not smaller than $r$. Thus
 $$
 L_2\geq (2+\pi)r_2 \geq (2+\pi)r
 $$
 by the theorem of Joris \cite{joris}. Consequently
$$
L\geq\sqrt{L_1^2+L_2^2}\geq \sqrt{(6r)^2+((2+\pi)r)^2}>7.90164\,r,
$$
which establishes \eqref{eq:2}.

 \subsection{The horizon}
 Here we develop some integral formulas needed to prove \eqref{eq:3}. For a point $x$ outside $\S^2$ consider the cone generated by all rays which emanate from $x$ and are tangent to $\S^2$. This cone touches $\S^2$ along a circle which we call the horizon of $x$. The \emph{horizon} of a curve $\gamma$, which we denote by $H(\gamma)$, is defined as the total area, counted with multiplicity, covered by horizons of all points of $\gamma$. Note that a point $p$ of $\S^2$ belongs to $H(\gamma)$ if and only if the tangent plane $T_p\S^2$ intersects $\gamma$. Thus
$$
H(\gamma):=\int_{p\in\S^2} \#\big(\gamma^{-1}(T_p\S^2)\big)\, dp.
$$
The closedness of $\gamma$ together with a bit of convexity theory, quickly  yields the following lower bound for the horizon. Recall that we say a curve $\gamma\colon[a,b]\to\R^3$ inspects the sphere $\S^2$, or is an inspections curve provided that it lies outside $\S^2$ and $\S^2$ lies in its convex hull.

\begin{lem}\label{lem:8pi}
If $\gamma$ is a closed inspection curve of $\S^2$, then 
\begin{equation}\label{eq:8pi}
8\pi\leq H(\gamma).
\end{equation}
\end{lem}
\begin{proof}
We claim that for every $p\in\S^2\setminus\gamma$,  $T_p\S^2$ intersects $\gamma$ in at least two points.
To see this set $C:=\conv(\gamma)$. Either $T_p\S^2$ is a support plane of $C$, or else $C$ has points in the interior of each of the closed half-spaces determined by $T_p\S^2$. In the latter case it is obvious that the claim holds. Suppose then that $T_p\S^2$ is a support plane of $C$. By Caratheodory's theorem \cite[p. 3]{schneider:book}, $p$ must lie in a line segment or a triangle $\Delta$ whose vertices belong to $\gamma$. Since, by assumption $p\not\in\gamma$, $p$ must belong to the relative interior of $\Delta$. Consequently $\Delta$ has to lie in $T_p\S^2$. Then the vertices of $\Delta$ yield the desired points.
\end{proof}

 Next we develop an analytic formula for computing $H$, following the basic outline of the proof of Crofton's formula in Chern \cite[p. 116]{chern:1967}.
Suppose that $\gamma$ is piecewise $\C^1$, and its projection into $\S^2$, $\ol\gamma:=\gamma/\|\gamma\|$ has non vanishing speed. Let $\ol T:=\ol\gamma'/\|\ol\gamma'\|$ and $\ol\nu:=\ol\gamma\times \ol T$. Then $(\ol\gamma,\ol T,\ol\nu)$ is a moving orthonormal frame along $\ol\gamma$. It is easy to check that the derivative of this frame is given by
$$
\left( 
 \begin{array}{c}
\ol\gamma  \\
\ol T  \\
\ol\nu \end{array} 
\right)'=
 \left( 
 \begin{array}{ccc}
0 & v & 0 \\
-v & 0 & \lambda \\
0 & -\lambda & 0 \end{array} 
\right)
\left( 
 \begin{array}{c}
\ol\gamma  \\
\ol T  \\
\ol\nu \end{array} 
\right),
$$
where $v:=\|\ol\gamma'\|$  and 
 $\lambda\colon[a,b]\to\R$ is some scalar function.
Define $F\colon[a,b]\times[0,2\pi]\to\S^2$ by
$$
F(t,\theta):=h(t)\ol\gamma(t)+r(t)\big( \cos(\theta) \ol T(t)+\sin(\theta)\ol\nu(t)\big),
$$
where
$$
r:=\frac{\sqrt{\|\gamma\|^2-1}}{\|\gamma\|}\quad\text{and}\quad h:=\frac{1}{\|\gamma\|}.
$$
For each $t\in[a,b]$, $F(t,\theta)$ parametrizes the horizon of $\gamma(t)$.  In particular, for all $p\in\S^2$, 
$$
F^{-1}(p)=\gamma^{-1}(T_p \S^2).
$$
Thus the area formula \cite[Thm 3.2.3]{federer:book} yields that
\begin{equation}\label{eq:horizon}
H(\gamma)=\int_{p\in\S^2}\#F^{-1}(p)\,dp=\int_a^b\int_0^{2\pi} \Jac(F)\,d\theta dt,
\end{equation}
where $\Jac(F):=\|\partial F/\partial t\times \partial F/\partial\theta\|$ denotes the Jacobian of $F$.
A computation shows that
$$
\Jac(F)=\left|r(t)v(t)\cos(\theta)-h'(t)\right|.
$$
Further we have
$$
v=\frac{\sqrt{\|\gamma'\|^2\|\gamma\|^2-\l \gamma,\gamma'\r^2}}{\|\gamma\|^2}\quad\text{and}\quad h'=-\frac{\l \gamma,\gamma'\r}{\|\gamma\|^3}.
$$
So we conclude that
$$
\Jac(F)=\frac{1}{\|\gamma\|^3}\left| \sqrt{(\|\gamma\|^2-1)(\|\gamma'\|^2\|\gamma\|^2-\l \gamma,\gamma'\r^2)}\cos(\theta)+\l \gamma,\gamma'\r \right|.
$$
Note that if $\|\gamma'\|=1$ and $\alpha(t)$ is the angle between $\gamma'(t)$ and $\gamma(t)$, then
$$
\cos(\alpha)=\frac{\l \gamma,\gamma'\r}{\|\gamma\|},\quad \text{and}\quad \sin(\alpha)=\frac{\sqrt{\|\gamma\|^2-\l \gamma,\gamma'\r^2}}{\|\gamma\|},
$$
which yields
\begin{equation*}\label{eq:jac}
\Jac(F)=\frac{1}{\|\gamma\|^2}\left| \sqrt{\|\gamma\|^2-1}\sin(\alpha)\cos(\theta)+\cos(\alpha) \right|.
\end{equation*}
The observations in this section may now be summarized as follows:

\begin{lem}\label{lem:horizonformula}
Let $\gamma\colon[a,b]\to\R^3$ be a piecewise $\C^1$ curve, and $\alpha(t)$ be the angle between $\gamma(t)$ and $\gamma'(t)$. Suppose that $\alpha(t)\neq 0$, $\pi$ except at finitely many points, and $\|\gamma'\|=1$. Then
\begin{equation}\label{eq:horizon2}
H(\gamma)=\int_a^b\int_0^{2\pi} \frac{1}{\|\gamma\|^2}\left| \sqrt{\|\gamma\|^2-1}\sin(\alpha)\cos(\theta)+\cos(\alpha) \right|\,d\theta dt.
\end{equation}
\qed
\end{lem}

\begin{note}\label{note:gammac}
If $\|\gamma\|=c$, then  \eqref{eq:8pi} together with Lemma \ref{lem:horizonformula} yield
$$
8\pi\leq H(\gamma)
= \frac{\sqrt{c^2-1}}{c^2}  4 L
\leq 2 L,
$$
and equality holds only if  $c=\sqrt{2}$.
Thus, as has been noted by Jean-Marc Schlenker  \cite{orourkeMO}, see also \cite[Sec. 2.2]{zalgaller:2003}:
If  $\gamma$  inspects $\S^2$ and $\|\gamma\|=c$, then 
$
L \geq 4\pi,
$
which is the optimal inequality for closed inspection curves  originally conjectured by Zalgaller \cite{zalgaller:1996}, and also suggested by Gjergji Zaimi \cite{orourkeMO}. We will improve this observation in Proposition \ref{prop:crofton} below.
\end{note}

 \subsection{The closed case}
 To prove \eqref{eq:3}, we begin by recording a pair of lemmas which yield an upper bound for the horizon. Let us say that a piecewise $\C^1$ curve $\gamma$ inspects the sphere $\S^2$ \emph{efficiently}, provided that that $\gamma$ inspects $\S^2$ and the tangent lines of $\gamma$ do not enter $\S^2$.  
 
 \begin{lem}\label{lem:tildegamma}
For every closed polygonal curve $\gamma$ which inspects $\S^2$, there is a closed polygonal curve $\tilde\gamma$, with $L[\tilde\gamma]\leq L[\gamma]$, which inspects $\S^2$ efficiently.
\end{lem}
\begin{proof}
Let $E$ be an edge of $\gamma$ whose corresponding line enters $\S^2$ (if $E$ does not exist, then there is nothing to prove). Let $p$ be the vertex of $E$ which is farthest from $\S^2$, and $C$ be the cone with vertex $p$ which is tangent to $\S^2$. Then the other vertex of $E$, say $p'$ belongs to the region $X$ which lies inside $C$ and outside $\S^2$, see Figure \ref{fig:cones}. Consider the polygonal arc $p' p$ of $\gamma$ which is different from $E$. Note that $\gamma$ cannot lie entirely in $X$ for then $\S^2$ cannot be in the convex hull of $\gamma$. So $p'p$ must have a point outside $X$. In particular, there is a point of $p'p$, other than $p$  which belongs to $\p C$. Let $q$ be the first such point, and replace the subsegment $p'q$ of $p'p$  with the line segment joining $p$ and $q$. This procedure removes $E$ and does not increase the number of edges of $\gamma$ or its length. Further, the new curve still inspects $\S^2$, because $p$ ``sees" all points of $\S^2$ which were visible from any points of $pq$. Since $\gamma$ has only finitely many edges, repeating this procedure eventually  yields the desired curve $\tilde \gamma$.
\end{proof}

\begin{figure}[h]
   \centering
    \begin{overpic}[height=1.1in]{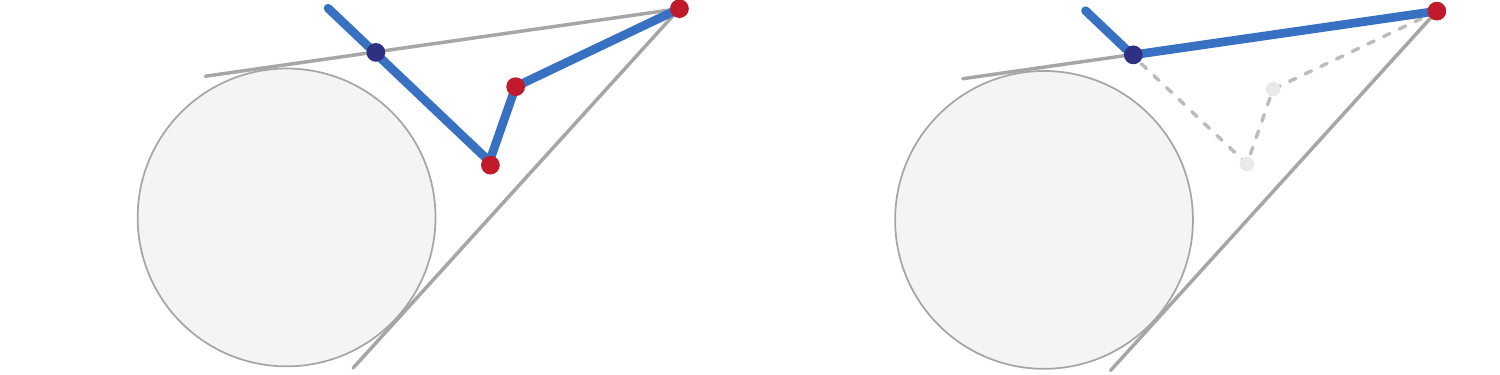} 
    \put(46.5,24.5){\small $p$}
     \put(32.5,20){\small $p'$}
    \put(25,23.5){\small $q$}
     \put(97,24.5){\small $p$}
    \put(76,23.5){\small $q$}
    \end{overpic}
   \caption{}
   \label{fig:cones}
\end{figure}

\begin{lem}\label{lem:alpha}
Suppose that $\gamma$ is a  piecewise $\C^1$ curve which inspects $\S^2$ efficiently, then
\begin{equation}\label{eq:JacF}
H(\gamma)\leq \frac{4\pi}{3\sqrt3}L.
\end{equation}
\end{lem}
\begin{proof}
By \eqref{eq:horizon}, it suffices to show that 
\begin{equation}\label{eq:JacF2}
\int_0^{2\pi} \Jac(F)\,d\theta \leq \frac{4\pi}{3\sqrt3}.
\end{equation}
To this end note that, by the Cauchy-Schwartz inequality,
\begin{eqnarray*}
\Jac(F)&=&\frac{1}{\|\gamma\|^2}\left|\Big\lan \big(\sqrt{\|\gamma\|^2-1}\cos(\theta),1\big), \big(\sin(\alpha),\cos(\alpha)\big)\Big \ran\right|.\\
&\leq& \frac{1}{\|\gamma\|^2}\|\big(\sqrt{\|\gamma\|^2-1}\cos(\theta),1\big)\| \leq \frac{1}{\|\gamma\|}.
\end{eqnarray*}
Thus \eqref{eq:JacF2} is satisfied whenever $\|\gamma\|\geq 3\sqrt{3}/2$. So  it suffices now to check \eqref{eq:JacF2} for $\|\gamma\|<3\sqrt{3}/2<2.6$.
To this end, set
$$
I(x,y):=\int_0^{2\pi} \frac{1}{x^2}\left| \sqrt{x^2-1}y\cos(\theta)+\sqrt{1-y^2} \right|\,d\theta.
$$
Then by \eqref{eq:horizon2}, $\int_0^{2\pi} \Jac(F)\,d\theta=I(\|\gamma\|,\sin(\alpha))$, because replacing $\cos(\alpha)$ with $|\cos(\alpha)|$ in \eqref{eq:horizon2} amounts at most to switching $\theta$ to $-\theta$, which does not affect the value of the integral.
Further note that, by elementary trigonometry, the tangent lines of $\gamma$ avoid the interior of $\S^2$ if and only if 
$$
\sin(\alpha)\geq \frac{1}{\|\gamma\|}.
$$
So we just need to check that $I\leq 4\pi/(3\sqrt3)\approx 2.4$ for $1\leq x\leq 3$ and $1/x\leq y\leq 1$, which may be done with the aid of a computer algebra system. In particular, graphing $I$ shows that the maximum of $I$ over the given region is achieved on the boundary curve $y=1/x$, see Figure \ref{fig:graph}.
\begin{figure}[h]
   \centering
    \begin{overpic}[height=1.75in]{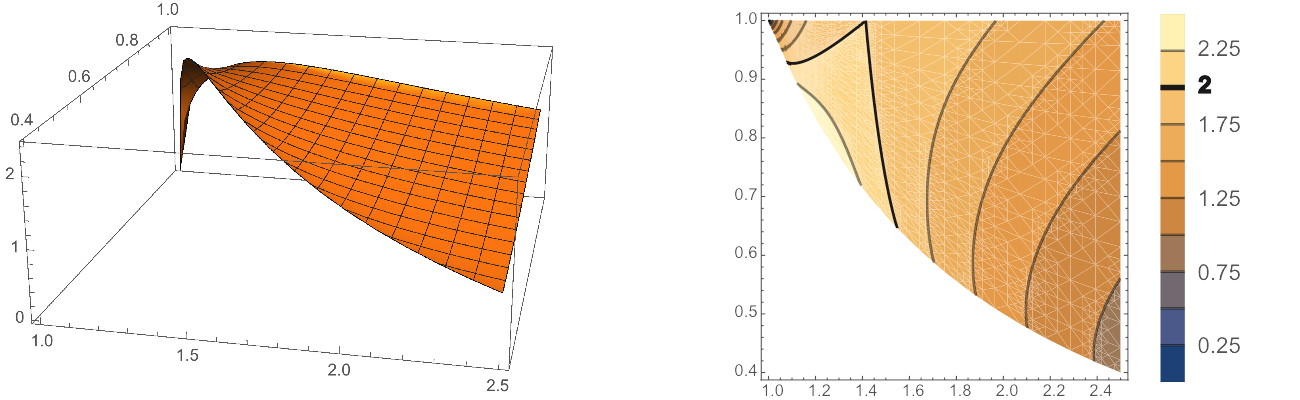} 
              \end{overpic}
   \caption{}
   \label{fig:graph}
\end{figure}
 Then it remains to note that
$$
I\(x,\frac1x\)= \frac{\sqrt{x^2-1}}{x^3}\int_0^{2\pi}( \cos(\theta)+1 )\,d\theta=2\pi \frac{\sqrt{x^2-1}}{x^3}\leq\frac{4\pi}{3\sqrt3},
$$
which completes the proof.
\end{proof}

Now we are ready to complete the proof of  \eqref{eq:3}. We may assume, after a translation, that $\S^2$ is a sphere of maximal radius which is contained in the convex hull of $\gamma$, and whose interior is disjoint from $\gamma$. In particular $r=1$.  We need to show then that $L\geq 6\sqrt 3$. To this end we may assume that
$\gamma$ is polygonal. Indeed,  let
$\gamma_i$ be a sequence of polygonal curves, converging to $\gamma$. Then $L_i/r_i\to L$ where $L_i$ and $r_i$ are the length and inradius of $\gamma_i$ respectively.  Thus, if $L_i/r_i\geq 6\sqrt3$, it follows that $L/r\geq 6\sqrt3$ as desired. 
Now we may let $\tilde\gamma$ be as in Lemma \ref{lem:tildegamma}. Then by Lemmas \ref{lem:8pi} and \ref{lem:alpha}
$$
8\pi\leq H(\tilde\gamma)\leq \frac{4\pi}{3\sqrt3}L[\tilde\gamma]\leq \frac{4\pi}{3\sqrt3}L,
$$
which completes the proof of Theorem \ref{thm:main2}.

\begin{note}\label{note:1.6}
 We were able to establish the conjectured sharp inequality $L\geq 4\pi$ only for the case of $\|\gamma\|=c$, since in this case $H(\gamma)\leq 2L$. If the same upper bound may be established for the class of all closed curves which inspect $\S^2$ efficiently, then we may replace the right hand side of the last displayed expression by $2L$, and thus obtain $L\geq 4\pi$ for all closed curves inspecting $\S^2$. The contour graph in Figure \ref{fig:graph} shows that $H(\gamma)\leq 2L$ if $\min\|\gamma\|\geq 1.6$. Thus, in this case $L\geq 4\pi$.
\end{note}

\section{Generalizations}\label{sec:general}

\subsection{More inradius estimates via Crofton}\label{sec:crofton}

Here we use Crofton's formulas to generalize the earlier observation in Note \ref{note:gammac}, on inspection curves of constant height:

\begin{prop}\label{prop:crofton}
Let $\gamma\colon[a,b]\to\R^3$ be a closed rectifiable curve which inspects $\S^2$, and set $M:=\max\|\gamma\|$, $m:=\min\|\gamma\|$. Then
$$
L\geq\frac{2\pi Mm}{\sqrt{M^2-1}}.
$$
In particular, when $M=m$, or $M\leq 2/\sqrt3$,  then $L\geq 4\pi$.
\end{prop}

Recall that, as we pointed out in Note \ref{note:1.6}, the conjectured inequality $L\geq 4\pi$ holds when $M\geq 1.6$. This, together with the above proposition shows that if there exists a closed inspection curve with $L<4\pi$, then $1.15\leq\|\gamma(t)\|\leq 1.6$ for some $t\in[a,b]$.
To establish the above inequality, let us record that:

\begin{lem}[Crofton-Blaschke-Santalo]\label{lem:crofton}
For every point $p\in\S^2$, and $0\leq\rho\leq\pi/2$, let $C_\rho(p)$ denote the circle of spherical radius $\rho$ centered at $p$. Then
$$
L=\frac{1}{4\sin(\rho)}\int_{p\in\S^2}\#\gamma^{-1}(C_\rho(p))
$$ 
\end{lem}

Crofton was the first person to obtain integrals  of this type for planar curves \cite{santalo:1976}. According to Santalo \cite{santalo:1942}, Blaschke observed the analogous phenomena on the sphere for regular curves \cite{blaschke:1937}, which were then extended to all rectifiable curves by Santalo \cite[(37)]{santalo:1942}.

\begin{proof}[Proof of Proposition \ref{prop:crofton}]
Let $\rho(t)$ be the (spherical) radius of the ``visibility circle", generated by rays which emanate from $\gamma(t)$ and are tangent to $\S^2$. Then $\cos(\rho)=1/\|\gamma\|\geq 1/M$ by simple trigonometry, which in turn yields that
$$
\sin(\rho)\leq\frac{\sqrt{M^2-1}}{M}.
$$
  Let $\ol\rho$ be the supremum of the  radii of the  visibility circles.  Then the union of all spherical disks of radius $\ol\rho$ centered at points of $\ol\gamma$ cover $\S^2$. Consequently, $\ol\gamma$ intersects all circles of radius $\ol\rho'$ in $\S^2$ at least twice for all $\ol\rho'>\ol\rho$. So,  by the Crofton formula, Lemma \ref{lem:crofton},
$$
\ol L\geq \frac{1}{4\sin(\ol\rho)}\int_{\S^2} 2\geq \frac{2\pi}{\sin(\ol\rho)}
$$ 
where $\ol L$ denotes the length of $\ol\gamma$. Finally note that 
$$
L\geq m\ol L,
$$
 since $m\ol L$ is the length of the projection of $\gamma$ into the sphere of radius $m$ centered at the origin, and $\|\gamma\|\geq m$. Combining the last three displayed expressions completes the proof.
\end{proof}

\subsection{Estimates for the $n^{th}$ inradius}\label{sec:nth}

The convex hull of a closed curve $\gamma$ in $\R^3$ coincides with the set of all points $p$ such that almost every plane through $p$ intersects $\gamma$ in at least $2$ points. Motivated by this phenomenon, the $n^{th}$ \emph{hull} of $\gamma$  has been defined \cite{CKKS} as the set of all points $p$ such that every plane through $p$ intersects $\gamma$  in at least $2n$ points. Accordingly, the $n^{th}$ \emph{inradius} $r_n$ of $\gamma$  may be defined as the supremum of the radii of all balls which are contained in the $n^{th}$ hull of $\gamma$  and do not intersect $\gamma$, which generalizes the notion of the inradius defined in the introduction. The proof of \eqref{eq:3}
may now be easily generalized as follows:

\begin{thm}\label{thm:nth}
For any closed rectifiable curve $\gamma\colon[a,b]\to\R^3$, 
$$
\frac{L}{r_n}\geq 6\sqrt 3\,n.
$$
\end{thm}
\begin{proof}
As in the proof of Theorem \ref{thm:main2}, we may assume that $r_n=1$, and $\S^2$ is a sphere of maximal radius contained in the $n^{th}$ hull of $\gamma$, and whose interior is disjoint from $\gamma$. Then every tangents plane $T_p\S^2$ intersects $\gamma$ at least $2n$ times, and consequently $H(\gamma)\geq 4\pi\times 2n$ by the definition of the horizon. On the other hand $H(\gamma)\leq 4\pi/(3\sqrt3)L$ by Lemma \ref{lem:alpha}. Thus
$$
8n\pi \leq H(\gamma)\leq \frac{4\pi}{3\sqrt3}L,
$$
which yields $L\geq 6\sqrt 3\,n$ as desired.
\end{proof}

The notion of $n^{th}$ hull is of interest in geometric knot theory, since it was established in \cite{CKKS} that knotted curves have nonempty second hulls.

\subsection{Estimates for width and inradius  in $\R^n$}\label{sec:Rn}

The generalized Wienholtz theorem (Corollary \ref{cor:genwienholtz}) together with the length decomposition lemma (Lemma \ref{lem:wienholtz}) quickly yield:

\begin{lem}
Let $\gamma\colon[a,b]\to\R^n$ be a rectifiable curve. Suppose that for all projections of $\gamma$ into hyperplanes of $\R^n$ we have $L/w\geq c_1$ and $L/r\geq c_2$. Then 
$$
L\geq\sqrt{c_1^2+9}\,w\quad\text{and}\quad L\geq\sqrt{c_2^2+36}\,r
$$
Further, if $\gamma$ is closed, and for all projections of $\gamma$ into hyperplanes of $\R^n$ we have $L/w\geq c_3$ and $L/r\geq c_4$, then 
$$
L\geq\sqrt{c_3^2+16}\,w\quad\text{and}\quad L\geq \sqrt{c_4^2+64}\,r.
$$\qed
\end{lem}

Thus we may inductively extend our estimates for the width and inradius problems, in Theorems \ref{thm:main1} and \ref{thm:main2}, to higher dimensions:

\begin{thm}
Let $\gamma\colon[a,b]\to\R^{2+k}$ be a rectifiable curve. Then
$$
L\geq\sqrt{2.2782^2+9k}\,w\quad\text{and}\quad L\geq\sqrt{(\pi+2)^2+36k}\,r.
$$
Further, if $\gamma$ is closed,
$$
L\geq\sqrt{\pi^2+16k}\,w\quad\text{and}\quad L\geq\sqrt{(6\sqrt{3})^2+64(k-1)}\,r.
$$\qed
\end{thm}

\section*{Acknowledgements}

The author thanks Joseph O'Rourke for posting the sphere inspection problem on MathOverflow \cite{orourkeMO} which provided the initial stimulus for this work. Thanks also to the commenters on MathOverflow, specially Gjergji Zaimi and Jean-Marc Schlenker, for useful  observations on this problem. Finally the author thanks Igor Belegradek for his translation of Zalgaller's $L_5$-conjecture, included in  Note \ref{note:L5}.

\bibliographystyle{abbrv}

\begin{thebibliography}{10}

\bibitem{adhikari&pitman}
A.~Adhikari and J.~Pitman.
\newblock The shortest planar arc of width {$1$}.
\newblock {\em Amer. Math. Monthly}, 96(4):309--327, 1989.

\bibitem{alexander:2009}
R.~Alexander.
\newblock The geometry of wide curves in the plane.
\newblock {\em J. Geom.}, 93(1-2):1--20, 2009.

\bibitem{ambrosio}
L.~Ambrosio and P.~Tilli.
\newblock {\em Topics on analysis in metric spaces}, volume~25 of {\em Oxford
  Lecture Series in Mathematics and its Applications}.
\newblock Oxford University Press, Oxford, 2004.

\bibitem{ayari-dubuc}
S.~Ayari and S.~Dubuc.
\newblock La formule de {C}auchy sur la longueur d'une courbe.
\newblock {\em Canad. Math. Bull.}, 40(1):3--9, 1997.

\bibitem{barbier:1860}
E.~Barbier.
\newblock Note sur le probl{\`e}me de l'aiguille et le jeu du joint couvert.
\newblock {\em Journal de math{\'e}matiques pures et appliqu{\'e}es}, pages
  273--286, 1860.

\bibitem{bellman:1956}
R.~Bellman.
\newblock A minimization problem.
\newblock {\em Bulletin of the AMS}, 62:270, 1956.

\bibitem{blaschke:1937}
W.~Blaschke.
\newblock {\em Vorlesungen {\"u}ber integralgeometrie}.
\newblock Number~20. BG Teubner, 1937.

\bibitem{bbi:book}
D.~Burago, Y.~Burago, and S.~Ivanov.
\newblock {\em A course in metric geometry}, volume~33 of {\em Graduate Studies
  in Mathematics}.
\newblock American Mathematical Society, Providence, RI, 2001.

\bibitem{CKKS}
J.~Cantarella, G.~Kuperberg, R.~B. Kusner, and J.~M. Sullivan.
\newblock The second hull of a knotted curve.
\newblock {\em Amer. J. Math.}, 125(6):1335--1348, 2003.

\bibitem{chan&golynski:2003}
T.~M. Chan, A.~Golynski, A.~Lopez-Ortiz, and C.-G. Quimper.
\newblock The asteroid surveying problem and other puzzles.
\newblock In {\em Proceedings of the Nineteenth Annual Symposium on
  Computational Geometry}, SCG '03, pages 372--373, New York, NY, USA, 2003.
  ACM.

\bibitem{chern:1967}
S.~S. Chern.
\newblock Curves and surfaces in euclidean space.
\newblock {\em Studies in Global Geometry and Analysis}, 4(1):967, 1967.

\bibitem{cfg:book}
H.~T. Croft, K.~J. Falconer, and R.~K. Guy.
\newblock {\em Unsolved problems in geometry}.
\newblock Springer-Verlag, New York, 1994.
\newblock Corrected reprint of the 1991 original [MR 92c:52001], Unsolved
  Problems in Intuitive Mathematics, II.

\bibitem{eggleston:1982}
H.~G. Eggleston.
\newblock The maximal inradius of the convex cover of a plane connected set of
  given length.
\newblock {\em Proc. London Math. Soc. (3)}, 45(3):456--478, 1982.

\bibitem{faber-mycielski:1986}
V.~Faber and J.~Mycielski.
\newblock The shortest curve that meets all the lines that meet a convex body.
\newblock {\em Amer. Math. Monthly}, 93(10):796--801, 1986.

\bibitem{federer:book}
H.~Federer.
\newblock {\em Geometric measure theory}.
\newblock Springer-Verlag New York Inc., New York, 1969.
\newblock Die Grundlehren der mathematischen Wissenschaften, Band 153.

\bibitem{finch&wetzel}
S.~R. Finch and J.~E. Wetzel.
\newblock Lost in a forest.
\newblock {\em Amer. Math. Monthly}, 111(8):645--654, 2004.

\bibitem{hiriart:2008}
J.-B. Hiriart-Urruty.
\newblock Du calcul diff{\'e}rentiel au calcul variationnel: un aper{\c{c}}u de
  l'{\'e}volution de p. fermata nos jours.
\newblock {\em Quadrature}, (70):8--18, 2008.

\bibitem{joris}
H.~Joris.
\newblock Le chasseur perdu dans la for\^et.
\newblock {\em Elem. Math.}, 35(1):1--14, 1980.
\newblock Un probl{\`e}me de g{\'e}om{\'e}trie plane.

\bibitem{klotzler:1987}
R.~Kl{\"o}tzler and S.~Pickenhain.
\newblock Universale rettungskurven ii.
\newblock {\em Zeitschrifte f{\"u}r Analysis und ihre Anwendungen}, 6:363--369,
  1987.

\bibitem{kusner&sullivan:distortion}
R.~B. Kusner and J.~M. Sullivan.
\newblock On distortion and thickness of knots.
\newblock In {\em Topology and geometry in polymer science (Minneapolis, MN,
  1996)}, pages 67--78. Springer, New York, 1998.

\bibitem{orourkeMO}
J.~O'Rourke.
\newblock Shortest closed curve to inspect a sphere (question posted on
  mathoverflow).
\newblock {\em
  mathoverflow.net/questions/69099/shortest-closed-curve-to-inspect-a-sphere},
  June 2011.

\bibitem{pelling}
M.~J. Pelling.
\newblock Classroom {N}otes: {F}ormulae for the {A}rc-{L}ength of a {C}urve in
  {$R^N$}.
\newblock {\em Amer. Math. Monthly}, 84(6):465--467, 1977.

\bibitem{santalo:1942}
L.~A. Santal{\'o}.
\newblock Integral formulas in {C}rofton's style on the sphere and some
  inequalities referring to spherical curves.
\newblock {\em Duke Math. J.}, 9:707--722, 1942.

\bibitem{santalo:1976}
L.~A. Santal{\'o}.
\newblock {\em Integral geometry and geometric probability}.
\newblock Addison-Wesley Publishing Co., Reading, Mass.-London-Amsterdam, 1976.
\newblock With a foreword by Mark Kac, Encyclopedia of Mathematics and its
  Applications, Vol. 1.

\bibitem{schneider:book}
R.~Schneider.
\newblock {\em Convex bodies: the {B}runn-{M}inkowski theory}.
\newblock Cambridge University Press, Cambridge, 1993.

\bibitem{sullivan:2008}
J.~M. Sullivan.
\newblock Curves of finite total curvature.
\newblock In {\em Discrete differential geometry}, volume~38 of {\em
  Oberwolfach Semin.}, pages 137--161. Birkh\"auser, Basel, 2008.

\bibitem{wienholtz:diameter}
D.~Wienholtz.
\newblock The smallest diameter of projections of closed curves into
  hyperplanes.
\newblock {\em Unpublished Manuscript}, 2000.

\bibitem{wiehnoltz:parallel}
D.~Wienholtz.
\newblock A special way how two planes can bound a given closed curve.
\newblock {\em Unpublished Manuscript}, 2000.

\bibitem{zalgaller:1961}
V.~A. Zalgaller.
\newblock How to get out of the woods.
\newblock {\em On a problem of Bellman (in Russian), Matematicheskoe
  Prosveshchenie}, 6:191--195, 1961.

\bibitem{zalgaller:1994}
V.~A. Zalgaller.
\newblock The problem of the shortest space curve of unit width.
\newblock {\em Mat. Fiz. Anal. Geom.}, 1(3-4):454--461, 1994.

\bibitem{zalgaller:1996}
V.~A. Zalgaller.
\newblock Extremal problems on the convex hull of a space curve.
\newblock {\em Algebra i Analiz}, 8(3):1--13, 1996.

\bibitem{zalgaller:2003}
V.~A. Zalgaller.
\newblock Shortest inspection curves for a sphere.
\newblock {\em Zap. Nauchn. Sem. S.-Peterburg. Otdel. Mat. Inst. Steklov.
  (POMI)}, 299(Geom. i Topol. 8):87--108, 328, 2003.

\end{thebibliography}

\end{document}